\documentclass[12pt,a4paper]{article}
\usepackage{}

\usepackage{latexsym}
\usepackage{amsmath}
\usepackage{amssymb}
\usepackage{arydshln}

\usepackage{cite}
\usepackage{stmaryrd}
\usepackage{enumerate}

\usepackage{hyperref}

\usepackage{color}
\usepackage{lineno}
\usepackage{graphicx}
\usepackage{ae}
\usepackage{amsmath}
\usepackage{amssymb}
\usepackage{latexsym}
\usepackage{url}
\usepackage{epsfig}
\usepackage{mathrsfs}
\usepackage{amsfonts}
\usepackage{amsthm}

\usepackage{float}
\usepackage{subfig}
\usepackage{tikz}
\usetikzlibrary{positioning}

\newcommand{\Cay}{\mathrm{Cay}}
\newcommand{\Gal}{\mathrm{Gal}}
\newcommand{\Aut}{\mathrm{Aut}}
\newcommand{\Inv}{\mathrm{Inv}}

\newcommand{\sym}{\mathrm{Sym}}

\newcommand{\D}{\Delta}

\newtheorem{theorem}{Theorem}[section]

\newtheorem{lemma}[theorem]{Lemma}
\newtheorem*{lem-no}{Lemma}

\newtheorem{cor}[theorem]{Corollary}

\newtheorem{example}{Example}

\theoremstyle{definition}

\numberwithin{equation}{section} 
\allowdisplaybreaks    

\def\qed{\hfill$\Box$\vspace{12pt}}

\long\def\delete#1{}

\usepackage{xcolor}
\usepackage[normalem]{ulem}

\begin{document}
\title {Algebraic degrees of $n$-Cayley digraphs over abelian groups}

\author{Hao Li$^{a,b,c}$,~Xiaogang Liu$^{a,b,c,}$\thanks{Supported by the National Natural Science Foundation of China (No. 12371358) and the Guangdong Basic and Applied Basic Research Foundation (No. 2023A1515010986).}~$^,$\thanks{ Corresponding author. Email addresses: haoli1111@mail.nwpu.edu.cn, xiaogliu@nwpu.edu.cn.}~
\\[2mm]
{\small $^a$School of Mathematics and Statistics,}\\[-0.8ex]
{\small Northwestern Polytechnical University, Xi'an, Shaanxi 710072, P.R.~China}\\
{\small $^b$Research \& Development Institute of Northwestern Polytechnical University in Shenzhen,}\\[-0.8ex]
{\small Shenzhen, Guandong 518063, P.R. China}\\
{\small $^c$Xi'an-Budapest Joint Research Center for Combinatorics,}\\[-0.8ex]
{\small Northwestern Polytechnical University, Xi'an, Shaanxi 710129, P.R. China}\\
}

\date{}

\openup 0.5\jot
\maketitle

\begin{abstract}
A digraph is called an $n$-Cayley digraph if its automorphism group has an $n$-orbit semiregular subgroup. We determine the splitting fields of $n$-Cayley digraphs over abelian groups and compute a bound on their algebraic degrees, before applying our results on Cayley digraphs over non-abelian groups.

\smallskip


\end{abstract}

\section{Introduction}\label{intro}
A \emph{directed graph} (abbreviated to \emph{digraph}) $\Gamma$ is an ordered pair $(V(\Gamma),A(\Gamma))$, where $V(\Gamma)$ is the set of vertices and $A(\Gamma)$ is the set of arcs. Each arc is associated with an ordered pair of vertices. If an arc $a$ is associated with $(u,v)$, then $a$ is said to be directed from $u$ to $v$. The \emph{adjacency matrix} of a digraph $\Gamma$ is the matrix $\mathcal{A}(\Gamma)=[a_{u,v}]$, where rows and columns of $\mathcal{A}(\Gamma)$ are both indexed by the vertex set of $\Gamma$ and $a_{u,v}$ is the number of arcs directed from $u$ to $v$. We refer the readers to \cite{bondyandmurty} for graph theory terminology. Let $F$ be a field. Denote by $F[x]$ the ring of polynomial in one variable over $F$. Let $f(x)\in F[x]$. We say that $f(x)$ \emph{splits} over an extension field $E$ of $F$ if $f(x)$ factors completely into linear factors in $E[x]$. An extension field $K$ of $F$ is called a \emph{splitting field} of $f(x)$ over $F$ if $f(x)$ splits over $K$ and $K$ is generated by the roots of $f(x)$. Accordingly, we say that a digraph $\Gamma$
\emph{splits} over an extension field $E$ of the rational number field $\mathbb{Q}$ if the characteristic polynomial of $\mathcal{A}(\Gamma)$ splits over $E$. And an extension field $K$ of $\mathbb{Q}$ is called a \emph{splitting field} of a digraph $\Gamma$ if $K$ is a splitting field of the characteristic polynomial of $\mathcal{A}(\Gamma)$ over $\mathbb{Q}$. By the fundamental theorem of algebra, all digraphs split in the complex number field $\mathbb{C}$. In this paper, we consider a splitting field of a digraph as a subfield of $\mathbb{C}$. A digraph $\Gamma$ is called an \emph{integral graph} if eigenvalues of $\mathcal{A}(\Gamma)$ are all integers. Since a complex number is an algebraic integer if and only if it is an eigenvalue of a matrix whose entries are all integers, eigenvalues of adjacency matrices of digraphs are algebraic integers. Moreover, since an algebraic integer is rational if and only if it is an integer, a digraph is integral if and only if it splits over $\mathbb{Q}$. The algebraic degree of a graph was initially introduced in \cite{aldeg} as a generalization of the integrality of a graph. Let $K$ be a splitting field of a digraph $\Gamma$. The \emph{algebraic degree} of $\Gamma$, denoted by $\deg(\Gamma)$, is defined as $$\deg(\Gamma)=[K:\mathbb{Q}],$$
where $[K:\mathbb{Q}]$ denotes the degree of the finite field extension $K$ of $\mathbb{Q}$, namely, the dimension of $K$ as a $\mathbb{Q}$-vector space.

Let $G$ be a group and let $S$ be a subset of $G$ without the identity. The \emph{Cayley digraph} $\Cay(G,S)$ is defined to be the digraph with vertex set $G$ and arcs drawn from $g\in G$ to $h\in G$ whenever $hg^{-1}\in S$. In particular, if $G$ is a cyclic group, then $\Cay(G,S)$ is called a \emph{circulant digraph}. Moreover, if $S$ is symmetric (that is, $\forall s\in S$, $s^{-1}\in S$), then $\Cay(G,S)$ is an undirected graph, called a \emph{Cayley graph}. By a theorem of Sabidussi \cite{sabidussi}, a digraph $\Gamma$ is a Cayley digraph over a group $G$ if and only if there exists a regular subgroup of $\Aut(\Gamma)$ isomorphic to $G$, where $\Aut(\Gamma)$ denotes the group of graph automorphisms of the digraph $\Gamma$. As a natural generalization of Sabidussi's theorem, $n$-Cayley digraphs were introduced. A digraph $\Gamma$ is called an \emph{$n$-Cayley digraph} over a group $G$ if there exists an $n$-orbit semiregular subgroup of $\Aut(\Gamma)$ isomorphic to $G$. An additional motivation for studies on $n$-Cayley graphs comes from the well-known Polycirculant conjecture \cite{polycirculant1,polycirculant2} regarding existence of semiregular automorphisms in vertex-transitive digraphs. In short, this conjecture states that every vertex-transitive digraph is an $n$-Cayley graph on a cyclic group. We refer the readers to \cite{polyresult1,polyresult2,polyresult3,polyresult4,polyresult5,polyresult6} for results on this conjecture. Most of the research related to $n$-Cayley digraphs has been done for the case $n=2$ so far \cite{semibi0,semispectrum,semispectrum2,bisemi1,bisemi2,bisemi3,semibi1,bisemi4,bisemi5,bisemi6}. Note that undirected $2$-Cayley graphs are also called semi-Cayley graphs or bi-Cayley graphs. Currently, there are relatively few results on $n$-Cayley digraphs in general \cite{ncayleylaplacian1,ncayleylaplacian2,tcayleycharacteristic,normality}.



Up until now, integral Cayley graphs have been studied extensively \cite{int1,dic1,dih,int2,int3,int4}. We refer the readers to the survey \cite[Section~3]{LiuZ2022} for more results on integral Cayley graphs. Splitting fields and algebraic degrees of Cayley graphs have attracted much attention over the past few years. In \cite{aldeg1,aldeg2}, M\"{o}nius studied the splitting fields and algebraic degrees of circulant graphs. In \cite{feili}, Li extended M\"{o}nius' results to Cayley digraphs over abelian groups. In \cite{mixed}, Huang et al. determined the splitting fields and algebraic degrees of mixed Cayley graphs over abelian groups. In \cite{mixed2}, Liu et al. studied the HS-splitting fields of abelian mixed Cayley graphs. In \cite{hyper}, Sripaisan et al. studied the algebraic degrees of Cayley hypergraphs. In \cite{lulu}, Lu et al. determined the splitting fields of Cayley graphs over abelian groups and dihedral groups and computed a bound on the algebraic degrees of Cayley graphs over dihedral groups. In \cite{dicycsemidihe}, Liu et al. determined the splitting fields of Cayley graphs over dicyclic and semi-dihedral groups and computed a bound on their algebraic degrees. In \cite{quasiabelian}, Wang et al. determined the splitting fields of quasi-abelian semi-Cayley digraphs and computed a bound on the algebraic degrees of them. In \cite{2cayley}, Wu et al. determined the splitting fields of $2$-Cayley digraphs over abelian groups and computed a bound on the algebraic degrees of them.

Motivated by the foregoing results, we study splitting fields and algebraic degrees of $n$-Cayley digraphs over abelian groups. As applications, we obtain results on Cayley digraphs over non-abelian groups. In \cite{aldeg2}, M\"{o}nius raised a question, whether the results on circulant graphs in \cite{aldeg2} can be extended to Cayley graphs over non-abelian groups. There have been results on Cayley digraphs whose vertex group has an abelian subgroup with index $2$ (See \cite{2cayley}). Typical cases include generalized dihedral groups, generalized dicyclic groups and semi-dihedral groups. Our work aims to further investigate the question raised by M\"{o}nius in \cite{aldeg2} as well.

Our paper is organized as follows. In Section $\ref{pre}$, we introduce some symbols, definitions and lemmas. In Section \ref{semicayley}, we present our main results. We determine splitting fields of $n$-Cayley digraphs over abelian groups and compute a bound on their algebraic degrees. In Section \ref{applications}, we apply our results in Section \ref{semicayley} on Cayley digraphs over non-abelian groups.



\section{Preliminaries}\label{pre}
In this section, we introduce some symbols, definitions and lemmas.


\subsection{Field and Galois theory}\label{FG}

The symbols, definitions and results concerning the field and Galois theory come from \cite{michaelartin,fieldandgalois,fieldandgalois2}.

\begin{itemize}
\item Let $F\subseteq K\subseteq E$ be finite field extensions. It is known (See\cite[Page~215, Theorem~4.2]{fieldandgalois}) that
\begin{equation}\label{bino}
[E:F]=[E:K][K:F].
\end{equation}
\item Let $F\subseteq E$ be a field extension. We refer to the pair $F\subseteq E$ as the field extension $E|F$. Let $S=\{\alpha_{1},\alpha_{2},\ldots,\alpha_{k}\}$ be a subset of $E$. Denote by $F(S)=F(\alpha_{1},\alpha_{2},\ldots,\alpha_{k})$ the subfield of $E$ generated by $F$ and $S=\{\alpha_{1},\alpha_{2},\ldots,\alpha_{k}\}$.






\item Let $F$ be a field. Let $K_{1}$ and $K_{2}$ be field extensions of $F$ contained in some common extension $E$ of $F$. The \emph{composite} of $K_{1}$ and $K_{2}$, denoted by $K_{1}K_{2}$, is the subfield of $E$ generated by $K_{1}$ and $K_{2}$. It is known (See\cite[Page~448, Corollary~15.3.8]{michaelartin}) that
    \begin{equation}
    \label{compositedegree}
    [K_{1}K_{2}:F]\leq[K_{1}:F][K_{2}:F].
    \end{equation}
\item Let $F$ be a field. Let $f(x)\in F[x]$ be a polynomial of degree $n$. Let $K$ be a splitting field of $f(x)$ over $F$. It is known (See\cite[Page~29, Corollary~6]{fieldandgalois2}) that
    \begin{equation}\label{splittingdegree}
    [K:F]\leq n!,
    \end{equation}
    where $n!$ is the factorial of $n$.

\item Let $E|F$ be a field extension. The \emph{Galois group} of $E$ over $F$, denoted by $\Gal(E|F)$, is defined as
\begin{equation}\label{gal}
\Gal(E|F)=\{\sigma\in\Aut(E):\forall x\in F,\;\sigma(x)=x\},
\end{equation}
where $\Aut(E)$ denotes the group of field automorphisms of the field $E$. Let $H$ be a subgroup of $\Gal(E|F)$. The \emph{$H$-fixed subfield} of $E$, denoted by $\Inv(H)$, is defined as
\begin{equation}\label{inv}
\Inv(H)=\{x\in E:\forall\sigma\in H,\;\sigma(x)=x\}.
\end{equation}
Let $H^{(1)}$ and $H^{(2)}$ be subgroups of $\Gal(E|F)$. By definition, if $H^{(1)}\subseteq H^{(2)}$, then
\begin{equation}\label{inverse1}
\Inv(H^{(1)})\supseteq\Inv(H^{(2)}).
\end{equation}
\item Let $E|F$ be a finite Galois extension (See\cite[Page~21, Definition~2.15]{fieldandgalois2} for the definition of Galois extension). Let $K$ be an intermediate field of $E|F$, that is, $F\subseteq K\subseteq E$. Let $H$ be a subgroup of $\Gal(E|F)$. The Fundamental Theorem of Galois Theory (See\cite[Page~239]{fieldandgalois}) states
\begin{equation}\label{invgal}
    \Inv(\Gal(E|K))=K,\;
    \Gal(E|\Inv(H))=H
\end{equation}
and
\begin{equation}\label{degree}
[K:F]=\frac{|\Gal(E|F)|}{|\Gal(E|K)|}, \;  [\Inv(H):F]=\frac{|\Gal(E|F)|}{|H|}.
\end{equation}

\item We refer to the imaginary unit as $\iota=\sqrt{-1}$. Let $N\geq2$ be an integer. Set $\omega_{N}=e^{\frac{2\pi\iota}{N}}$. $\omega_{N}$ is an $N$-th root of unity. $\mathbb{Q}(\omega_{N})|\mathbb{Q}$ is a cyclotomic extension (See\cite[Page~72, Definition~7.1]{fieldandgalois2} for the definition of cyclotomic extension). Moreover, the cyclotomic extension $\mathbb{Q}(\omega_{N})|\mathbb{Q}$ is known (See\cite[Page~72, Proposition~7.2]{fieldandgalois2}) to be Galois. For convenience, set
$$
[N]=\{1,2,\ldots,N\}.
$$
It is known (See\cite[Page~75, Corollary~7.8]{fieldandgalois2}) that
$$
\Gal(\mathbb{Q}(\omega_{N})|\mathbb{Q})=\{\sigma_{j}:j\in[N],
\;\mathrm{gcd}(j,N)=1\},
$$
where $\sigma_{j}$ is determined by
$$\sigma_{j}(\omega_{N})=\omega_{N}^{j}$$
and $\mathrm{gcd}(j,N)$ denotes the greatest common divisor of $j$ and $N$. Note that
\begin{equation}\label{phin}
|\Gal(\mathbb{Q}(\omega_{N})|\mathbb{Q})|=\phi(N),
\end{equation}
where $\phi$ is the \emph{Euler's totient function} given by
\begin{align*}
\phi(k)=\left\{
\begin{array}{ll}
k\prod_{i=1}^{s}(1-\frac{1}{p_{i}}), & \text{if $k$ has prime factorisation $p_{1}^{r_{1}}p_{2}^{r_{2}}\cdots p_{s}^{r_{s}}\geq2$},\\[0.2cm]
0, & \text{if $k=1$},
\end{array}
\right.
\end{align*}
for any integer $k\geq1$.
\end{itemize}
\subsection{Representation theory of finite groups}\label{RT}

The symbols, definitions and results concerning the representation theory of finite groups come from \cite{representationtheory}.

\begin{itemize}
\item
Let $K$ be a field. Let $G$ be a finite group. Denote by $K[G]$ the collection of mappings from $G$ to $K$. Namely,
$$
K[G]=\{a\mid a:G\rightarrow K\}.
$$
Let $S\subseteq G$. The symbol $\delta_{S}\in K[G]$ is defined as
$$
\delta_{S}(g)=
\left\{
\begin{array}{ll}
1, & \text{if~}g\in S,\\[0.2cm]
0, & \text{if~}g\notin S.
\end{array}
\right.
$$
Let $a,b\in K[G]$. The addition $a+b:G\rightarrow K$ is defined by
$$
(a+b)(g)=a(g)+b(g),\quad\forall g\in G.
$$
The \emph{convolution} $a*b:G\rightarrow K$ is defined by
$$
a*b(g)=\sum_{h\in G}a(gh^{-1})b(h),\quad\forall g\in G.
$$
$K[G]$ is known (See\cite[Page~408]{fieldandgalois}) to be an \emph{algebra} over $K$, which indicates two facts about $K[G]$. Firstly, $(K[G],+,*)$ is a ring. Secondly, $K[G]$ is a $K$-vector space whose scalar multiplication is defined by
$$(k\cdot a)(g)=ka(g),\quad\forall k\in K,\;\forall a\in K[G],\;\forall g\in G.$$
The algebra $K[G]$ is called the \emph{group algebra} over $K$ of the finite group $G$.


\item Let $G$ be a finite group. Denote by $\mathbb{C}[G]$ the group algebra over $\mathbb{C}$ of $G$. It is known (See\cite[Page~54, Proposition~5.2.4]{representationtheory}) that $(\mathbb{C}[G],+,*)$ is a commutative ring if $G$ is abelian. In the following, $G$ is assumed to be an abelian group. The collection of irreducible characters of $G$, denoted by $\widehat{G}$, is called the \emph{dual group} of $G$. For $a\in\mathbb{C}[G]$ and $\chi\in\widehat{G}$, define
\begin{equation}\label{fourierco}
\widehat{a}(\chi)=\sum_{g\in G}a(g)\overline{\chi(g)},
\end{equation}
where $\overline{\;\cdot\;}$ denotes the complex conjugate. The complex numbers $\widehat{a}(\chi)$ for $\chi\in\widehat{G}$ are often called the \emph{Fourier coefficients} of $a$. For $a,b\in\mathbb{C}[G]$ and $\chi\in\widehat{G}$, we have
\begin{equation}\label{addition}
\widehat{a+b}(\chi)=\widehat{a}(\chi)+\widehat{b}(\chi).
\end{equation}
It is known (See\cite[Page 56, Theorem 5.3.6]{representationtheory}) that for $a\in\mathbb{C}[G]$ and $g\in G$,
\begin{equation}\label{fourierinv}
a(g)=\frac{1}{|G|}\sum_{\chi\in\widehat{G}}\widehat{a}(\chi)\chi(g),
\end{equation}
which is called the \emph{Fourier inversion}. Fourier transform is known (See\cite[Page~56, Theorem~5.3.8]{representationtheory}) to be a ring isomorphism, which means that for $a,b\in\mathbb{C}[G]$ and $\chi\in\widehat{G}$,
\begin{equation}\label{convolution}
\widehat{a*b}(\chi)=\widehat{a}(\chi)\widehat{b}(\chi).
\end{equation}
\end{itemize}

\begin{itemize}
\item In this paper, a ring is always assumed to contain a unit element. Let $R$ be a commutative ring. Let $n\geq1$ be an integer. Denote by $M_{n}(R)$ the \emph{ring of $n\times n$ matrices over the ring $R$}. Let $A=[a_{i,j}]_{n\times n}\in M_{n}(R)$. The \emph{determinant} of $A$, denoted by $\det\langle A\rangle$, is defined as
\begin{equation}\label{det2}
\det\langle A\rangle=\sum_{\sigma\in\sym_{[n]}}(-1)^{\mathrm{sgn}(\sigma)}\prod_{x\in[n]}a_{x,\sigma(x)},
\end{equation}
where
\begin{align*}
\mathrm{sgn}(\sigma)=\left\{
\begin{array}{ll}
1, & \text{if $\sigma$ is even},\\[0.2cm]
-1, & \text{if $\sigma$ is odd}.
\end{array}
\right.
\end{align*}
Since $\det\langle A\rangle$ is a sum of products of elements in the ring $R$, it is natural that
\begin{equation}\label{detcontained}
\det\langle A\rangle\in R.
\end{equation}
Let $X=\{\xi_{1},\xi_{2},\ldots,\xi_{k}\}$ be a subset of $[n]$ with $\xi_{1}<\xi_{2}<\cdots<\xi_{k}$. Denote by $A|_{X}$ the submatrix of $A$ obtained by striking out rows and columns indexed by $[n]\backslash X$, that is, $A|_{X}=[a_{\xi_{i},\xi_{j}}]_{k\times k}$. $A|_{X}$ is called a \emph{principal minor} of $A$ of order $k$. Denote by $U_{k}$ the collection of subsets of $[n]$ of cardinality $k$, that is,
$$
U_{k}=\{X\subseteq[n]:|X|=k\}.
$$
For $k\in[n]$, define
\begin{equation}\label{minordef}
\beta_{k}\langle A\rangle=\sum_{X\in U_{k}}\det\langle A|_{X}\rangle.
\end{equation}
Note that $\forall k\in[n]$, $\beta_{k}\langle A\rangle$ is a sum of determinants. By (\ref{detcontained}), for $k\in[n]$, it is natural that
\begin{equation}\label{betacontained}
\beta_{k}\langle A\rangle\in R.
\end{equation}
In addition, define $\beta_{0}\langle A\rangle=1$.


\item Let $n\geq1$ be an integer. Let $A\in M_{n}(\mathbb{C})$. Denote the characteristic polynomial of $A$ by
    $$\mathcal{P}\langle A\rangle(x)=\det\langle xI_{n}-A\rangle\in\mathbb{C}[x],$$
    where $I_{n}$ is the identity matrix of order $n$. It is known (See\cite[Page~196, below equation~(33)]{fieldandgalois}) that
    \begin{equation}\label{defcharacteristic}
    \mathcal{P}\langle A\rangle(x)=\sum_{k=0}^{n}(-1)^{k}\beta_{k}\langle A\rangle x^{n-k}.
    \end{equation}
    For a polynomial $f(x)\in\mathbb{C}[x]$, denote the collection of distinct roots of $f(x)$ by $\mathcal{R}(f)$, that is,
    $$
    \mathcal{R}(f)=\{x\in\mathbb{C}:f(x)=0\}.
    $$

\item Let $n\geq1$ be an integer. Let $G$ be a finite abelian group. Let $A=[a_{i,j}]_{n\times n}\in M_{n}(\mathbb{C}[G])$. Note that $A$ is a matrix whose entries are mappings from $G$ to $\mathbb{C}$ and that
    \begin{equation}\label{convolproddef}
    \det\langle A\rangle=\sum_{\sigma\in\sym_{[n]}}(-1)^{\mathrm{sgn}(\sigma)}\prod^{*}_{x\in[n]}a_{x,\sigma(x)}
    \end{equation}
    is a mapping from $G$ to $\mathbb{C}$. The symbol $\prod\limits^{*}$ in the above equation represents the convolution of a sequence of functions in $\mathbb{C}[G]$. For $\chi\in\widehat{G}$, define
    \begin{equation}\label{widetildedef}
    \widetilde{A}(\chi)=[\widehat{a_{i,j}}(\chi)]_{n\times n}.
    \end{equation}
Note that $\widetilde{A}(\chi)$ is a matrix whose entries are complex numbers. For a subset $X=\{\xi_{1},\xi_{2},\ldots,\xi_{k}\}$ of $[n]$ with $\xi_{1}<\xi_{2}<\cdots<\xi_{k}$, we have
\begin{equation}\label{AXAX}
\widetilde{A|_{X}}(\chi)=\widetilde{[a_{\xi_{i},\xi_{j}}]_{k\times k}}(\chi)=[\widehat{a_{\xi_{i},\xi_{j}}}(\chi)]_{k\times k}=\widetilde{A}(\chi)|_{X}.
\end{equation}
\end{itemize}





\begin{lemma}\label{det}
Let $n\geq1$ be an integer. Let $G$ be a finite abelian group. Let $A=[a_{i,j}]_{n\times n}\in M_{n}(\mathbb{C}[G])$. Then $\forall\chi\in\widehat{G}$,
$$\widehat{\det\langle A\rangle}(\chi)=\det\langle\widetilde{A}(\chi)\rangle.$$


\end{lemma}
\begin{proof}
For any $\chi\in\widehat{G}$,
\begin{align*}
\widehat{\det\langle A\rangle}(\chi)
&=\sum_{g\in G}\det\langle A\rangle(g)\overline{\chi(g)}\tag{by (\ref{fourierco})}
\\&=\sum_{g\in G}(\sum_{\sigma\in\sym_{[n]}}(-1)^{\mathrm{sgn}(\sigma)}\prod^{*}_{x\in[n]}a_{x,\sigma(x)})(g)\overline{\chi(g)}\tag{by (\ref{convolproddef})}
\\&=\sum_{g\in G}\sum_{\sigma\in\sym_{[n]}}(-1)^{\mathrm{sgn}(\sigma)}(\prod^{*}_{x\in[n]}a_{x,\sigma(x)})(g)\overline{\chi(g)}
\\&=\sum_{\sigma\in\sym_{[n]}}(-1)^{\mathrm{sgn}(\sigma)}\sum_{g\in G}(\prod^{*}_{x\in[n]}a_{x,\sigma(x)})(g)\overline{\chi(g)}
\\&=\sum_{\sigma\in\sym_{[n]}}(-1)^{\mathrm{sgn}(\sigma)}\widehat{\prod_{x\in[n]}a_{x,\sigma(x)}}(\chi)\tag{by (\ref{fourierco})}
\\&=\sum_{\sigma\in\sym_{[n]}}(-1)^{\mathrm{sgn}(\sigma)}\prod_{x\in[n]}\widehat{a_{x,\sigma(x)}}(\chi)\tag{by (\ref{convolution})}
\\&=\det\langle\widetilde{A}(\chi)\rangle.\tag{by (\ref{det2})}
\end{align*}
This completes the proof.\qed
\end{proof}
\begin{cor}\label{spm}
Let $n\geq1$ be an integer. Let $G$ be a finite abelian group. Let $A=[a_{i,j}]_{n\times n}\in M_{n}(\mathbb{C}[G])$. Then $\forall k\in[n]$, $\forall\chi\in\widehat{G}$,
$$\widehat{\beta_{k}\langle A\rangle}(\chi)=\beta_{k}\langle\widetilde{A}(\chi)\rangle.$$

\end{cor}

\begin{proof}
Note that $\beta_{k}\langle A\rangle\in\mathbb{C}[G]$. For any $k\in[n]$ and any $\chi\in\widehat{G}$,
\begin{align*}
\widehat{\beta_{k}\langle A\rangle}(\chi)
&=\widehat{(\sum_{X\in U_{k}}\det\langle A|_{X}\rangle)}(\chi)\tag{by (\ref{minordef})}
\\&=\sum_{X\in U_{k}}\widehat{\det\langle A|_{X}\rangle}(\chi)\tag{by (\ref{addition})}
\\&=\sum_{X\in U_{k}}\det\langle\widetilde{A|_{X}}(\chi)\rangle\tag{by Lemma \ref{det}}
\\&=\sum_{X\in U_{k}}\det\langle\widetilde{A}(\chi)|_{X}\rangle\tag{by (\ref{AXAX})}
\\&=\beta_{k}\langle\widetilde{A}(\chi)\rangle.\tag{by (\ref{minordef})}
\end{align*}
This completes the proof.\qed
\end{proof}


\subsection{Main lemmas}\label{ML}
We introduce some symbols and lemmas which play an important role in investigating the splitting fields of $n$-Cayley digraphs over abelian groups.

Let $N\geq2$ be an integer. The \emph{ring of integers modulo $N$} is denoted by $\mathbb{Z}_{N}$. Denote the group of units of $\mathbb{Z}_{N}$ by $\mathbb{Z}_{N}^{*}$. Through this work, when a finite abelian group $G$ is given, we will frequently use some symbols which simply depends on $G$. For the sake of convenience, we give the definitions of them here, once and for all. Henceforth, given a finite abelian group $G$, we will use these symbols with no further explanations for them.

Let $G$ be a finite abelian group. Set
$$
G=\mathbb{Z}_{d_{1}}\times\mathbb{Z}_{d_{2}}\times\cdots\times\mathbb{Z}_{d_{r}},
$$
where $\times$ denotes the \emph{Cartesian product}. Set
$$
N=|G|=\prod_{i=1}^{r}d_{i}.
$$
For any field automorphism $\sigma\in\Gal(\mathbb{Q}(\omega_{N})|\mathbb{Q})$, define $l_{\sigma}\in\mathbb{Z}_{N}^{*}$ by
\begin{equation}\label{l}
\sigma(\omega_{N})=\omega_{N}^{l_{\sigma}}.
\end{equation}
Besides, for $\sigma\in\Gal(\mathbb{Q}(\omega_{N})|\mathbb{Q})$, set a group automorphism
\begin{equation}\label{eta}
\begin{split}
\eta_{\sigma}: ~&G\rightarrow G,\\
&(g_{1},g_{2},\ldots,g_{r})\mapsto(l_{\sigma}\cdot g_{1},l_{\sigma}\cdot g_{2},\ldots,l_{\sigma}\cdot g_{r}),
\end{split}
\end{equation}
where $l_{\sigma}\cdot g_{i}$ denotes the multiplication modulo $d_{i}$ for $i=1,2,\ldots,r$. One may verify that
\begin{align*}
\eta:~&\Gal(\mathbb{Q}(\omega_{N})|\mathbb{Q})\rightarrow\Aut(G)
\\&\sigma\mapsto\eta_{\sigma}
\end{align*}
is a group action. Let $v=(v_{1},v_{2},\ldots,v_{r})\in G$. For $i=1,2,\ldots,r$, set a mapping
\begin{equation}\label{cyclic}
\begin{split}
\chi_{v}^{(i)}:~&\mathbb{Z}_{d_{i}}\rightarrow\mathbb{C},\\
&g_{i}\mapsto\omega_{d_{i}}^{v_{i}g_{i}}.
\end{split}
\end{equation}
Set
\begin{equation}\label{abelian}
\begin{split}
\chi_{v}:~&G\rightarrow\mathbb{C},
\\&(g_{1},g_{2},\ldots,g_{r})\mapsto\prod_{i=1}^{r}\chi_{v}^{(i)}(g_{i}).
\end{split}
\end{equation}
It is known (See\cite[Page~47, Proposition~4.5.1]{representationtheory} and \cite[Page~55, Example~5.3.3]{representationtheory}) that the dual group
$$\widehat{G}=\{\chi_{v}:v\in G\}.$$
By (\ref{cyclic}) and (\ref{abelian}), $\forall\chi\in\widehat{G}$, $\forall g\in G$, one may verify that
\begin{equation}\label{negative}
\chi(-g)=\overline{\chi(g)},
\end{equation}
where $-g$ denotes the inverse of $g$ in $G$.

\begin{lemma}\label{etasigma}
Let $G$ be a finite abelian group. $\forall \sigma\in \Gal(\mathbb{Q}(\omega_{N})|\mathbb{Q})$, $\forall v= (v_{1},v_{2},\ldots,v_{r})$ $\in G$, $\forall  g=(g_{1},g_{2},\ldots,g_{r})\in G$, we have $\sigma(\chi_{v}(g))=\chi_{v}(\eta_{\sigma}(g))=\chi_{\eta_{\sigma}(v)}(g)$.
\end{lemma}
\begin{proof}
$\forall\sigma\in\Gal(\mathbb{Q}(\omega_{N})|\mathbb{Q})$, $\forall v=(v_{1},v_{2},\ldots,v_{r})\in G$, $\forall g=(g_{1},g_{2},\ldots,g_{r})\in G$,
\begin{align*}
\sigma(\chi_{v}(g))
&=\sigma(\prod_{s=1}^{r}\chi_{v}^{(s)}(g_{s}))\tag{by (\ref{abelian})}
\\
&=\prod_{s=1}^{r}\sigma(\chi_{v}^{(s)}(g_{s}))\tag{$\sigma$ being an automorphism}
\\
&=\prod_{s=1}^{r}\sigma(\omega_{d_{s}}^{v_{s}g_{s}})\tag{by (\ref{cyclic})}
\\
&=\prod_{s=1}^{r}\omega_{d_{s}}^{v_{s}g_{s}l_{\sigma}}\tag{by (\ref{l})}
\\
&=\prod_{s=1}^{r}\omega_{d_{s}}^{v_{s}[\eta_{\sigma}(g)]_{s}}\tag{by (\ref{eta})}
\\
&=\prod_{s=1}^{r}\chi_{v}^{(s)}([\eta_{\sigma}(g)]_{s})\tag{by (\ref{cyclic})}
\\
&=\chi_{v}(\eta_{\sigma}(g))\tag{by (\ref{abelian})}
\end{align*}
and
\begin{align*}
\sigma(\chi_{v}(g))
&=\sigma(\prod_{s=1}^{r}\chi_{v}^{(s)}(g_{s}))\tag{by (\ref{abelian})}
\\
&=\prod_{s=1}^{r}\sigma(\chi_{v}^{(s)}(g_{s}))\tag{$\sigma$ being an automorphism}
\\
&=\prod_{s=1}^{r}\sigma(\omega_{d_{s}}^{v_{s}g_{s}})\tag{by (\ref{cyclic})}
\\
&=\prod_{s=1}^{r}\omega_{d_{s}}^{v_{s}g_{s}l_{\sigma}}\tag{by (\ref{l})}
\\
&=\prod_{s=1}^{r}\omega_{d_{s}}^{[\eta_{\sigma}(v)]_{s}g_{s}}\tag{by (\ref{eta})}
\\
&=\prod_{s=1}^{r}\chi_{\eta_{\sigma}(v)}^{(s)}(g_{s})\tag{by (\ref{cyclic})}
\\
&=\chi_{\eta_{\sigma}(v)}(g)\tag{by (\ref{abelian})},
\end{align*}
where $[\eta_{\sigma}(g)]_{s}$ and $[\eta_{\sigma}(v)]_{s}$ denote the $s$-th entries of $\eta_{\sigma}(g)$ and $\eta_{\sigma}(v)$, respectively.\qed
\end{proof}


\begin{lemma}\label{main}
Let $G$ be a finite abelian group. Let $K$ be a subfield of $\mathbb{Q}(\omega_{N})$. Set $H=\Gal(\mathbb{Q}(\omega_{N})|K)$. Let $a\in K[G]$. Then $\forall\chi\in\widehat{G}$, $\widehat{a}(\chi)\in K$ if and only if $\forall g\in G$, $\forall\sigma\in H$, $a(g)=a(\eta_{\sigma}(g))$.


\end{lemma}

\begin{proof}
We give the proof of the necessity first. Set $\sigma\in H$ and $g\in G$. Since $a(g)\in K$,
\begin{align*}
a(g)
&=\sigma(a(g))\tag{by (\ref{gal})}
\\
&=\sigma(\frac{1}{N}\sum_{v\in G}\widehat{a}(\chi_{v})\chi_{v}(g))\tag{by (\ref{fourierinv})}
\\&=\sigma(\frac{1}{N})\sum_{v\in G}\sigma(\widehat{a}(\chi_{v}))\sigma(\chi_{v}(g))\tag{$\sigma$ being an automorphism}
\\
&
=\frac{1}{N}\sum_{v\in G}\sigma(\widehat{a}(\chi_{v}))\sigma(\chi_{v}(g))\tag{$\frac{1}{n}\in\mathbb{Q}\subseteq K$}.
\end{align*}
Since $\forall\chi\in\widehat{G}$, $\widehat{a}(\chi)\in K$, we have 
\begin{align*}
\frac{1}{N}\sum_{v\in G}\sigma(\widehat{a}(\chi_{v}))\sigma(\chi_{v}(g))
&=\frac{1}{N}\sum_{v\in G}\widehat{a}(\chi_{v})\sigma(\chi_{v}(g))\tag{by (\ref{gal})}
\\&=\frac{1}{N}\sum_{v\in G}\widehat{a}(\chi_{v})\chi_{v}(\eta_{\sigma}(g))\tag{by Lemma \ref{etasigma}}
\\&=a(\eta_{\sigma}(g))\tag{by (\ref{fourierinv})}.
\end{align*}
Combining the above two equations, we have $a(g)=a(\eta_{\sigma}(g))$. This completes the proof of the necessity.

We give the proof of the sufficiency now. Let $\{H_{j}:j\in J\}$ be a complete set of distinct orbits of $G$ under $H$. Since $\forall g\in G$, $\forall\sigma\in H$, $a(g)=a(\eta_{\sigma}(g))$, set
\begin{equation}\label{proofmain1}
a=\sum_{j\in J}k_{j}\delta_{H_{j}},
\end{equation}
where $\forall j\in J$, $k_{j}\in K$. Then $\forall\sigma\in H$, $\forall\chi\in\widehat{G}$,
\begin{align*}
\sigma(\widehat{a}(\chi))&=\sigma(\sum_{g\in G}a(g)\overline{\chi(g)})\tag{by (\ref{fourierco})}
\\&=\sigma(\sum_{j\in J}\sum_{g\in H_{j}}a(g)\overline{\chi(g)})\tag{partitioning $G$ into orbits}
\\&=\sigma(\sum_{j\in J}\sum_{g\in H_{j}}k_{j}\overline{\chi(g)})\tag{by (\ref{proofmain1})}
\\&=\sum_{j\in J}\sum_{g\in H_{j}}\sigma(k_{j})\sigma(\overline{\chi(g)})\tag{$\sigma$ being an automorphism}
\\&=\sum_{j\in J}k_{j}\sum_{g\in H_{j}}\sigma(\overline{\chi(g)})\tag{by (\ref{gal})}
\\&=\sum_{j\in J}k_{j}\sum_{g\in H_{j}}\sigma(\chi(-g))\tag{by (\ref{negative})}
\\&=\sum_{j\in J}k_{j}\sum_{g\in H_{j}}\chi(\eta_{\sigma}(-g))\tag{by Lemma \ref{etasigma}}
\\&=\sum_{j\in J}k_{j}\sum_{g\in H_{j}}\overline{\chi(-\eta_{\sigma}(-g))}\tag{by (\ref{negative})}
\end{align*}
Recalling the definition of $\eta_{\sigma}$ in (\ref{eta}), we have $\eta_{\sigma}(g)=-\eta_{\sigma}(-g)$. Thus, $$\sum_{j\in J}k_{j}\sum_{g\in H_{j}}\overline{\chi(-\eta_{\sigma}(-g))}=\sum_{j\in J}k_{j}\sum_{g\in H_{j}}\overline{\chi(\eta_{\sigma}(g))}.$$
Since $\forall j\in J$, $H_{j}$ is an orbit and hence $\eta_{\sigma}$ permutes $H_{j}$, we have
\begin{equation}\label{lem25eq1}
\sum_{g\in H_{j}}\overline{\chi(\eta_{\sigma}(g))}=\sum_{g\in H_{j}}\overline{\chi(g)}.
\end{equation}
Thus,
\begin{align*}
\sum_{j\in J}k_{j}\sum_{g\in H_{j}}\overline{\chi(\eta_{\sigma}(g))}
&=\sum_{j\in J}k_{j}\sum_{g\in H_{j}}\overline{\chi(g)}\tag{by (\ref{lem25eq1})}
\\&=\sum_{j\in J}\sum_{g\in H_{j}}k_{j}\overline{\chi(g)}
\\&=\sum_{j\in J}\sum_{g\in H_{j}}a(g)\overline{\chi(g)}\tag{by (\ref{proofmain1})}
\\&=\sum_{g\in G}a(g)\overline{\chi(g)}\tag{uniting the orbits}
\\&=\widehat{a}(\chi)\tag{by (\ref{fourierco})}.
\end{align*}
Combining the above three equations, we have $\forall\sigma\in H$, $\forall\chi\in\widehat{G}$, $\sigma(\widehat{a}(\chi))=\widehat{a}(\chi)$. Thus, by $(\ref{inv})$ and $(\ref{invgal})$, $\forall\chi\in\widehat{G}$, $\widehat{a}(\chi)\in\Inv(H)=K$.\qed
\end{proof}
If $G$ is cyclic, Lemma \ref{main} is Theorem 4.2, together with Theorem 4.4 of \cite{aldeg2}, which generalizes the classification of all integral circulant graphs given by So in \cite{so}. In comparison to \cite{aldeg2}, we use the notion of Fourier coefficients instead of eigenvalues of Cayley digraphs. We refer the readers to \cite[Page~62, Theorem~5.4.10]{representationtheory} for the relation between Fourier coefficients and eigenvalues of Cayley digraphs.

\begin{lemma}\label{subgroup}
Let $G$ be a finite abelian group. Let $a\in\mathbb{Q}[G]$. Set $H=\{\sigma\in\Gal(\mathbb{Q}(\omega_{N})|\mathbb{Q}):\forall g\in G$, $a(g)=a(\eta_{\sigma}(g))\}$. Then $H$ is a subgroup of $\Gal(\mathbb{Q}(\omega_{N})|\mathbb{Q})$.
\end{lemma}
\begin{proof}
Trivially, the identity $id\in H$ and hence $H\neq\emptyset$. Let $\sigma_{1},\sigma_{2}\in H$. $\forall g\in G$,
$$a(g)=a(\eta_{id}(g))=a(\eta_{\sigma_{2}}(\eta_{\sigma_{2}^{-1}}(g))).$$
Set $g'=\eta_{\sigma_{2}^{-1}}(g)$. Then $$a(\eta_{\sigma_{2}}(\eta_{\sigma_{2}^{-1}}(g)))=a(\eta_{\sigma_{2}}(g')).$$
Since $\sigma_{2}\in H$,
$$a(\eta_{\sigma_{2}}(g'))=a(g').$$
Since $\sigma_{1}\in H$,
$$a(g')=a(\eta_{\sigma_{1}}(g'))=a(\eta_{\sigma_{1}}(\eta_{\sigma_{2}^{-1}}(g)))=a(\eta_{\sigma_{1}\sigma_{2}^{-1}}(g)).$$
Combining the above four equations, we have $a(g)=a(\eta_{\sigma_{1}\sigma_{2}^{-1}}(g))$. Hence $\sigma_{1}\sigma_{2}^{-1}\in H$,
which means that $H$
is a subgroup of $\Gal(\mathbb{Q}(\omega_{N})|\mathbb{Q})$.\qed
\end{proof}




\begin{lemma}\label{orbitchar}
Let $G$ be a finite abelian group. Let $a\in \mathbb{Q}[G]$. Set $H=\{\sigma\in\Gal(\mathbb{Q}(\omega_{N})|\mathbb{Q}):\forall g\in G,\;a(g)=a(\eta_{\sigma}(g))\}$. Then $\forall v\in G$, $\forall\sigma\in H$, $\widehat{a}(\chi_{v})=\widehat{a}(\chi_{\eta_{\sigma}(v)})$.


\end{lemma}

\begin{proof}
$\forall v\in G$, $\forall\sigma\in H$,
\begin{align*}
\widehat{a}(\chi_{\eta_{\sigma}(v)})
&=\sum_{g\in G}a(g)\overline{\chi_{\eta_{\sigma}(v)}(g)}\tag{by (\ref{fourierco})}
\\&=\sum_{g\in G}a(g)\overline{\chi_{v}(\eta_{\sigma}(g))}\tag{by Lemma \ref{etasigma}}
\\&=\sum_{g\in G}a(\eta_{\sigma}(g))\overline{\chi_{v}(\eta_{\sigma}(g))}\tag{by the definition of $H$}
\\&=\sum_{g\in G}a(g)\overline{\chi_{v}(g)}\tag{$\eta_{\sigma}$ permuting $G$}
\\&=\widehat{a}(\chi_{v})\tag{by (\ref{fourierco})}.
\end{align*}
This completes the proof. \qed
\end{proof}

\section{Algebraic degrees of $n$-Cayley digraphs over abelian groups}\label{semicayley}
In this section, we present our main results after introducing some lemmas.
\begin{lemma}\emph{(See\cite[Lemma~2]{tcayleycharacteristic})}\label{tcayleyrepresentation}
Let $\Gamma$ be a digraph. Let $G$ be a group. Then $\Gamma$ is an $n$-Cayley digraph over $G$ if and only if there exist subsets $S_{i,j}$ of $G$, where $1\leq i,j\leq n$, such that $\Gamma$ is isomorphic to the digraph $\Gamma'$ with
$$
V(\Gamma')=G\times\{1,2,\ldots,n\},\quad A(\Gamma')=\bigcup_{1\leq i,j\leq n}\{(g,i),(sg,j):g\in G,\;s\in S_{i,j}\}.
$$
\end{lemma}
Lemma \ref{tcayleyrepresentation} is a well-known structure representation of $n$-Cayley digraphs. By Lemma \ref{tcayleyrepresentation}, an $n$-Cayley digraph is characterized by a group $G$ and $n^{2}$ subsets $S_{i,j}$ of $G$. Hence, we denote an $n$-Cayley digraph by $\Gamma=\Cay(G;S_{i,j}|1\leq i,j\leq n)$. For convenience, when an $n$-Cayley digraph $\Gamma=\Cay(G;S_{i,j}|1\leq i,j\leq n)$ is given, we tacitly set $\Delta=[\delta_{-S_{j,i}}]_{n\times n}\in M_{n}(\mathbb{Q}[G])$, where $\delta_{-S_{j,i}}$ is the entry in the $i$-th row and the $j$-th column. By (\ref{betacontained}), $\forall k\in[n]$,
\begin{equation}\label{betarational}
\beta_{k}\langle\Delta\rangle\in\mathbb{Q}[G].
\end{equation}




The following result is a reduced version of \cite[Theorem~6]{tcayleycharacteristic}, which concerns $n$-Cayley digraphs over any finite groups. Here, we merely consider abelian groups.


\begin{lemma}\emph{(See\cite[Theorem~6]{tcayleycharacteristic})}\label{tcaypoly}
Let $\Gamma=\Cay(G;S_{i,j}|1\leq i,j\leq n)$ be an $n$-Cayley digraph over a finite abelian group $G$. Then $\mathcal{P}\langle\mathcal{A}(\Gamma)\rangle=\prod_{\chi\in\widehat{G}}\mathcal{P}\langle\widetilde{\D}(\chi)\rangle$.
\end{lemma}
In \cite{tcayleycharacteristic}, $\widetilde{\D}(\chi)$ is written as $[\sum_{g\in S_{j,i}}\chi(g)]_{i,j}$. To dispel confusions, we show $\widetilde{\D}(\chi)=[\sum_{g\in S_{j,i}}\chi(g)]_{i,j}$ here.
\begin{align*}
\widetilde{\D}(\chi)&=[\widehat{\delta_{-S_{j,i}}}(\chi)]_{i,j}\tag{by (\ref{widetildedef})}
\\&=[\sum_{g\in G}\delta_{-S_{j,i}}(g)\overline{\chi(g)}]_{i,j}\tag{by (\ref{fourierco})}
\\&=[\sum_{g\in G}\delta_{-S_{j,i}}(g)\chi(-g)]_{i,j}\tag{by (\ref{negative})}
\\&=[\sum_{g\in G}\delta_{-S_{j,i}}(-g)\chi(g)]_{i,j}
\\&=[\sum_{g\in G}\delta_{S_{j,i}}(g)\chi(g)]_{i,j}
\\&=[\sum_{g\in S_{j,i}}\chi(g)]_{i,j}.
\end{align*}
Lemma \ref{tcaypoly} determines the characteristic polynomial of an $n$-Cayley digraph, which provides a preliminary characterization of the splitting field of an $n$-Cayley digraph $\Gamma$. Namely, if $K$ is the splitting field of $\Gamma$, then
\begin{align*}
K&=\mathbb{Q}(\mathcal{R}(\mathcal{P}\langle\mathcal{A}(\Gamma)\rangle))\tag{by definition}
\\&=\mathbb{Q}(\bigcup_{v\in G}\mathcal{R}(\mathcal{P}\langle\widetilde{\D}(\chi_{v})\rangle)),\tag{by Lemma \ref{tcaypoly}}
\end{align*}
which we label as
\begin{equation}\label{preliminarycharacterization}
K=\mathbb{Q}(\bigcup_{v\in G}\mathcal{R}(\mathcal{P}\langle\widetilde{\D}(\chi_{v})\rangle)).
\end{equation}

\begin{lemma}\label{th1onlyif}
Let $\Gamma=\Cay(G;S_{i,j}|1\leq i,j\leq n)$ be an $n$-Cayley digraph over a finite abelian group $G$. Let $K$ be a subfield of $\mathbb{C}$. If $\Gamma$ splits over $K$, then $\forall k\in[n]$, $\forall g\in G$, $\forall\sigma\in\Gal(\mathbb{Q}(\omega_{N})|K\cap\mathbb{Q}(\omega_{N}))$, $\beta_{k}\langle\D\rangle(g)=\beta_{k}\langle\D\rangle(\eta_{\sigma}(g))$.



\end{lemma}


\begin{proof}
By definition, $\Gamma$ splitting over $K$ means that $\mathcal{P}\langle\mathcal{A}(\Gamma)\rangle$ splits over $K$. By Lemma \ref{tcaypoly}, $\forall\chi\in\widehat{G}$, $\mathcal{P}\langle\widetilde{\D}(\chi)\rangle$ splits over $K$. Note that if a polynomial splits in a field, then the coefficients of this polynomial are contained in this field. By (\ref{defcharacteristic}) and Corollary \ref{spm}, $\forall k\in[n]$, $\forall\chi\in\widehat{G}$,
\begin{equation}\label{th1onlyifeq1}
\widehat{\beta_{k}\langle\D\rangle}(\chi)\in K.
\end{equation}
By (\ref{cyclic}) and (\ref{abelian}), $\forall\chi\in\widehat{G}$,
\begin{equation}\label{th1onlyifeq3}
\chi\in\mathbb{Q}(\omega_{N})[G].
\end{equation}
Note that $\mathbb{Q}(\omega_{N})$ is a field and that $\mathbb{Q}(\omega_{N})$ is closed under complex conjugation. By (\ref{fourierco}), (\ref{betarational}) and (\ref{th1onlyifeq3}), $\forall k\in[n]$, $\forall\chi\in\widehat{G}$,
$$
\widehat{\beta_{k}\langle\D\rangle}(\chi)\in\mathbb{Q}(\omega_{N}).
$$
Recalling (\ref{th1onlyifeq1}), we have $\forall k\in[n]$, $\forall\chi\in\widehat{G}$,
$$
\widehat{\beta_{k}\langle\D\rangle}(\chi)\in K\cap\mathbb{Q}(\omega_{N}).
$$
By Lemma \ref{main}, $\forall k\in[n]$, $\forall g\in G$, $\forall\sigma\in\Gal(\mathbb{Q}(\omega_{N})|K\cap\mathbb{Q}(\omega_{N}))$,
$$\beta_{k}\langle\D\rangle(g)=\beta_{k}\langle\D\rangle(\eta_{\sigma}(g)).$$
This completes the proof.\qed

\end{proof}

Now we present our main results. We aim to compute a bound on the algebraic degree of an $n$-Cayley digraph $\Gamma$. (\ref{preliminarycharacterization}) already implies a bound on $\deg(\Gamma)$, that is,
$$
1\leq\deg(\Gamma)\leq(n!)^{N}.
$$
Nevertheless, we give a characterization of the splitting field of $\Gamma$ other than (\ref{preliminarycharacterization}) before computing the bound on $\deg(\Gamma)$.

\begin{theorem}\label{th1deg}
Let $\Gamma=\Cay(G;S_{i,j}|1\leq i,j\leq n)$ be an $n$-Cayley digraph over a finite abelian group $G$. Set $H=\{\sigma\in\Gal(\mathbb{Q}(\omega_{N})|\mathbb{Q}):\forall k\in[n],\;\forall g\in G,\;\beta_{k}\langle\D\rangle(g)=\beta_{k}\langle\D\rangle(\eta_{\sigma}(g))\}$ and $K_{0}=\Inv(H)$. Let $\{H_{j}:j\in J\}$ be the complete set of distinct orbits of $G$ under $H$ with the action being $\eta$. Let $\{v^{(j)}:j\in J\}$ be a complete set of representatives of $\{H_{j}:j\in J\}$. Then the splitting field of $\Gamma$ is $K_{0}(\bigcup_{j\in J}\mathcal{R}(\mathcal{P}\langle\widetilde{\D}(\chi_{v^{(j)}})\rangle))$ and the algebraic degree of $\Gamma$ is bounded by
$$
\frac{\phi(N)}{|H|}\leq\deg(\Gamma)\leq(n!)^{|J|}\frac{\phi(N)}{|H|}.
$$
\end{theorem}

\begin{proof}
Let $K$ be the splitting field of $\Gamma$. The proof of $K=K_{0}(\bigcup_{j\in J}\mathcal{R}(\mathcal{P}\langle\widetilde{\D}(\chi_{v^{(j)}})\rangle))$ is given in 2 parts. Firstly, we prove that $K\supseteq K_{0}$. By Lemma \ref{th1onlyif}, $\forall k\in[n]$, $\forall g\in G$, $\forall\sigma\in\Gal(\mathbb{Q}(\omega_{N})|K\cap\mathbb{Q}(\omega_{N}))$, $\beta_{k}\langle\D\rangle(g)=\beta_{k}\langle\D\rangle(\eta_{\sigma}(g))$ and so
\begin{equation}\label{th1degeq1}
\Gal(\mathbb{Q}(\omega_{N})|K\cap\mathbb{Q}(\omega_{N}))\subseteq H.
\end{equation}
Then
\begin{align*}
K&\supseteq K\cap\mathbb{Q}(\omega_{N})
\\&=\Inv(\Gal(\mathbb{Q}(\omega_{N})|K\cap\mathbb{Q}(\omega_{N})))\tag{by (\ref{invgal})}
\\&\supseteq\Inv(H)\tag{by (\ref{inverse1}) and (\ref{th1degeq1})}
\\&=K_{0},
\end{align*}
which we label as
\begin{equation}\label{th1degeq7}
K\supseteq K_{0}.
\end{equation}

Secondly, we prove that $\bigcup_{v\in G}\mathcal{R}(\mathcal{P}\langle\widetilde{\D}(\chi_{v})\rangle)
=\bigcup_{j\in J}\mathcal{R}(\mathcal{P}\langle\widetilde{\D}(\chi_{v^{(j)}})\rangle)
$. $\forall k\in[n]$, $\forall j\in J$, $\forall v\in H_{j}$, since $v$ and $v^{(j)}$ are on the same orbit, we may set $\sigma_{0}\in H$ such that
\begin{equation}\label{th1degeq3}
v^{(j)}=\eta_{\sigma_{0}}(v),
\end{equation}
and then we have
\begin{align*}
\widehat{\beta_{k}\langle\D\rangle}(\chi_{v})
&=\widehat{\beta_{k}\langle\D\rangle}(\chi_{\eta_{\sigma_{0}}(v)})\tag{by (\ref{betarational}) and Lemma \ref{orbitchar}}
\\&=\widehat{\beta_{k}\langle\D\rangle}(\chi_{v^{(j)}})\tag{by (\ref{th1degeq3})}.
\end{align*}
By Corollary \ref{spm}, $\forall k\in[n]$, $\forall j\in J$, $\forall v\in H_{j}$,
$$
\beta_{k}\langle\widetilde{\Delta}(\chi_{v})\rangle=\beta_{k}\langle\widetilde{\Delta}(\chi_{v^{(j)}})\rangle.
$$
Hence, by (\ref{defcharacteristic}),
\begin{equation}\label{th1degeq6}
\mathcal{P}\langle\widetilde{\D}(\chi_{v})\rangle
=\mathcal{P}\langle\widetilde{\D}(\chi_{v^{(j)}})\rangle.
\end{equation}
Then we have
\begin{align*}
\bigcup_{v\in G}\mathcal{R}(\mathcal{P}\langle\widetilde{\D}(\chi_{v})\rangle)
&=\bigcup_{j\in J}\bigcup_{v\in H_{j}}\mathcal{R}(\mathcal{P}\langle\widetilde{\D}(\chi_{v})\rangle)\tag{partitioning $G$ into orbits}
\\&=\bigcup_{j\in J}\bigcup_{v\in H_{j}}\mathcal{R}(\mathcal{P}\langle\widetilde{\D}(\chi_{v^{(j)}})\rangle)\tag{by (\ref{th1degeq6})}
\\&=\bigcup_{j\in J}\mathcal{R}(\mathcal{P}\langle\widetilde{\D}(\chi_{v^{(j)}})\rangle),
\end{align*}
which we label as
\begin{equation}\label{th1degeq8}
\bigcup_{v\in G}\mathcal{R}(\mathcal{P}\langle\widetilde{\D}(\chi_{v})\rangle)
=\bigcup_{j\in J}\mathcal{R}(\mathcal{P}\langle\widetilde{\D}(\chi_{v^{(j)}})\rangle).
\end{equation}

Now we have a characterization of the splitting field of $\Gamma$.
\begin{align*}
K
&=\mathbb{Q}(\bigcup_{v\in G}\mathcal{R}(\mathcal{P}\langle\widetilde{\D}(\chi_{v})\rangle))\tag{by (\ref{preliminarycharacterization})}
\\&=K_{0}(\bigcup_{v\in G}\mathcal{R}(\mathcal{P}\langle\widetilde{\D}(\chi_{v})\rangle))\tag{by (\ref{th1degeq7})}
\\&=K_{0}(\bigcup_{j\in J}\mathcal{R}(\mathcal{P}\langle\widetilde{\D}(\chi_{v^{(j)}})\rangle)
)\tag{by (\ref{th1degeq8})}
\end{align*}
which we label as
\begin{equation}\label{th1degeq10}
K=K_{0}(\bigcup_{j\in J}\mathcal{R}(\mathcal{P}\langle\widetilde{\D}(\chi_{v^{(j)}})\rangle)
).
\end{equation}

Finally, we compute a bound on $\deg(\Gamma)$. First, we prove that $\frac{\phi(N)}{|H|}\leq\deg(\Gamma)$.
\begin{align*}
\deg(\Gamma)
&=[K:\mathbb{Q}]
\\&=[K:K_{0}][K_{0}:\mathbb{Q}]\tag{by (\ref{bino})}
\\&\geq[K_{0}:\mathbb{Q}]
\\&=[\Inv(H):\mathbb{Q}]
\\&=\frac{|\Gal(\mathbb{Q}(\omega_{N})|\mathbb{Q})}{|H|}\tag{by (\ref{degree})}
\\&=\frac{\phi(N)}{|H|}.\tag{by (\ref{phin})}
\end{align*}
By the way, note that
\begin{equation}\label{th1degeq11}
[K_{0}:\mathbb{Q}]=\frac{\phi(N)}{|H|}.
\end{equation}

Next, we prove that $\deg(\Gamma)\leq(n!)^{|J|}\frac{\phi(N)}{|H|}$. By the definition of $H$, $\forall k\in[n]$, $\forall g\in G$, $\forall\sigma\in H$,
$$
\beta_{k}\langle\Delta\rangle(g)=\beta_{k}\langle\Delta\rangle(\eta_{\sigma}(g)).
$$
Note that by (\ref{betarational}), $\beta_{k}\langle\Delta\rangle\in\mathbb{Q}[G]\subseteq K_{0}[G]$ and that by (\ref{invgal}), $H=\Gal(\mathbb{Q}(\omega_{N}|K_{0})$. Then by Lemma \ref{main}, $\forall k\in[n]$, $\forall v\in G$,
\begin{equation}\label{th1degeq14}
\widehat{\beta_{k}\langle\D\rangle}(\chi_{v})\in K_{0}.
\end{equation}
Therefore, $\forall v\in G$,
\begin{align*}
\mathcal{P}\langle\widetilde{\D}(\chi_{v})\rangle
&=\sum_{k=0}^{n}(-1)^{k}\beta_{k}\langle\widetilde{\D}(\chi_{v})\rangle x^{n-k}\tag{by (\ref{defcharacteristic})}
\\&=\sum_{k=0}^{n}(-1)^{k}\widehat{\beta_{k}\langle\D\rangle}(\chi_{v})x^{n-k}\tag{by Corollary \ref{spm}}
\\&\in K_{0}[x].\tag{by (\ref{th1degeq14})}
\end{align*}
By (\ref{splittingdegree}), $\forall v\in G$,
\begin{equation}\label{th1degeq15}
[K_{0}(\mathcal{R}(\mathcal{P}\langle\widetilde{\D}(\chi_{v})\rangle)):K_{0}]\leq n!.
\end{equation}
Then we have
\begin{align*}
[K:K_{0}]&=[K_{0}(\bigcup_{j\in J}\mathcal{R}(\mathcal{P}\langle\widetilde{\D}(\chi_{v}^{(j)})\rangle)):K_{0}]\tag{by (\ref{th1degeq10})}
\\&\leq\prod_{j\in J}[K_{0}(\mathcal{R}(\mathcal{P}\langle\widetilde{\D}(\chi_{v}^{(j)})\rangle)):K_{0}]\tag{by (\ref{compositedegree})}
\\&\leq\prod_{j\in J}n!\tag{by (\ref{th1degeq15})}
\\&=(n!)^{|J|},
\end{align*}
which we label as
\begin{equation}\label{th1degeq16}
[K:K_{0}]\leq(n!)^{|J|}.
\end{equation}
At last,
\begin{align*}
\deg(\Gamma)&=[K:\mathbb{Q}]
\\&=[K:K_{0}][K_{0}:\mathbb{Q}]\tag{by (\ref{bino})}
\\&=[K:K_{0}]\frac{\phi(N)}{|H|}\tag{by (\ref{th1degeq11})}
\\&\leq(n!)^{|J|}\frac{\phi(N)}{|H|}.\tag{by (\ref{th1degeq16})}
\end{align*}
This completes the proof.\qed

\end{proof}

\section{Applications}\label{applications}
In \cite{aldeg2}, M\"{o}nius obtained results on circulant graphs and asked whether their results could be extended to Cayley graphs over non-abelian groups. In \cite{lulu}, Lu et al. obtained results on Cayley graphs over abelian groups. In this section, we apply our results in Section \ref{semicayley} on Cayley digraphs over non-abelian groups.

The following result is a slight modification of \cite[Lemma 8]{tcayleycharacteristic}.

\begin{lemma}\label{abeliansub}\emph{(See \cite[Lemma 8]{tcayleycharacteristic})}
Let $\Gamma=\Cay(G_{0},S)$ be a Cayley digraph. Let $G$ be a subgroup of $G_{0}$ with index $n$. If $\{s_{1},s_{2},\ldots,s_{n}\}$ is a left transversal to $G$ in $G_{0}$, then $\Gamma\cong\Cay(G;S_{i,j}|1\leq i,j\leq n)$, where $S_{i,j}=\{g\in G: s_{j}gs_{i}^{-1}\in S\}$.
\end{lemma}
\begin{proof}
Set $\Sigma=\Cay(G;S_{i,j}|1\leq i,j\leq n)$. Since $\{s_{1},s_{2},\ldots,s_{n}\}$ is a left transversal to $G$ in $G_{0}$, every element of $G$ is uniquely expressible in the form $s_{i}g$ with $g\in G$ and $i\in[n]$. Set a mapping
\begin{align*}
\psi:&G_{0}\rightarrow G\times[n]
\\&s_{i}g\mapsto(g,i).
\end{align*}
$\psi$ is clearly a bijection. For $s_{i}g_{1},s_{j}g_{2}\in G_{0}$,
\begin{align*}
&(s_{i}g_{1},s_{j}g_{2})\in A(\Gamma)
\\\Leftrightarrow&s_{j}g_{2}g_{1}^{-1}s_{i}^{-1}\in S
\\\Leftrightarrow&g_{2}g_{1}^{-1}\in S_{i,j}
\\\Leftrightarrow&((g_{1},i),(g_{2},j))\in A(\Sigma).
\end{align*}
Therefore, $\psi$ is a graph isomorphism from $\Gamma$ to $\Sigma$. Hence, $\Gamma\cong\Sigma$.\qed
\end{proof}
By Lemma \ref{abeliansub}, Theorem \ref{degsemidirectproduct} is a direct result from Theorem \ref{th1deg}.
\begin{theorem}\label{degsemidirectproduct}
Let $\Gamma=\Cay(G_{0},S)$ be a Cayley digraph. Let $G$ be a subgroup of $G_{0}$ with index $n$. Let $\{s_{1},s_{2},\ldots,s_{n}\}$ be a left transversal to $G$ in $G_{0}$. For $i,j\in[n]$, set $S_{i,j}=\{g\in G: s_{j}gs_{i}^{-1}\in S\}$. Set $H=\{\sigma\in\Gal(\mathbb{Q}(\omega_{N})|\mathbb{Q}):\forall k\in[n],\;\forall g\in G,\;\beta_{k}\langle\D\rangle(g)=\beta_{k}\langle\D\rangle(\eta_{\sigma}(g))\}$ and $K_{0}=\Inv(H)$. Let $\{H_{j}:j\in J\}$ be the complete set of distinct orbits of $G$ under $H$ with the action being $\eta$. Let $\{v^{(j)}:j\in J\}$ be a complete set of representatives of $\{H_{j}:j\in J\}$. Then the splitting field of $\Gamma$ is $K_{0}(\bigcup_{j\in J}\mathcal{R}(\mathcal{P}\langle\widetilde{\D}(\chi_{v^{(j)}})\rangle))$ and the algebraic degree of $\Gamma$ is bounded by
$$
\frac{\phi(N)}{|H|}\leq\deg(\Gamma)\leq(n!)^{|J|}\frac{\phi(N)}{|H|}.
$$
\end{theorem}

In particular, if $G_{0}=G$ is abelian, then the splitting field of $\Gamma$ is $K_{0}$ and $\deg(\Gamma)=\frac{\phi(N)}{|H|}$, which is the result on Cayley digraphs over abelian groups given in \cite{lulu}. We give an example in the following so as to illustrate how Theorem \ref{degsemidirectproduct} helps in considering splitting fields and algebraic degrees of Cayley digraphs.

\begin{example}
Notations in this example corresponds to those in Theorem \ref{degsemidirectproduct}. Let $\Gamma=\Cay(\mathbb{Z}_{7}\rtimes_{\varphi}\mathbb{Z}_{3},S)$ be a Cayley digraph. Set $G=\{(g,0):g\in\mathbb{Z}_{7}\}$. $\varphi:\mathbb{Z}_{3}\rightarrow\mathbb{Z}_{7}^{*}$ is a homomorphism, given by
\begin{align*}
\varphi_{k}:&\mathbb{Z}_{7}\rightarrow\mathbb{Z}_{7},
\\&g\mapsto2^{k}g,
\end{align*}
for $k\in\mathbb{Z}_{3}$. $S=\{(5,0),(6,0),(2,1),(3,1),(1,2),(4,2)\}$ is a subset of $\mathbb{Z}_{7}\rtimes_{\varphi}\mathbb{Z}_{3}$. Set $s_{1}=(0,1)$, $s_{2}=(0,2)$ and $s_{3}=(0,0)$. Then $\{s_{1},s_{2},s_{3}\}$ is a left transversal to $G$ in $\mathbb{Z}_{7}\rtimes_{\varphi}\mathbb{Z}_{3}$. By calculation,
\begin{equation*}
\Delta=
\begin{pmatrix}
&\delta_{\{1,4\}} &\delta_{\{3,5\}} & \delta_{\{2,6\}}\\
&\delta_{\{1,3\}} &\delta_{\{2,4\}} & \delta_{\{5,6\}}\\
&\delta_{\{3,6\}} &\delta_{\{4,5\}} & \delta_{\{1,2\}}
\end{pmatrix}.
\end{equation*}
Moreover, $\beta_{1}\langle\Delta\rangle=2\delta_{1,2,4}$, $\beta_{2}\langle\Delta\rangle=\delta_{3,5,6}-\delta_{1,2,4}$ and $\beta_{3}\langle\Delta\rangle=0$. Therefore $H=\{1,2,4\}$ and $\{H_{j}:j\in J\}=\{\{0\},\{1,2,4\},\{3,5,6\}\}$. And so by Theorem \ref{degsemidirectproduct}, we now have a bound on $\deg(\Gamma)$, that is,
$$2\leq\deg(\Gamma)\leq(3!)^{3}2.$$
Also by Theorem \ref{degsemidirectproduct}, the splitting field of $\Gamma$ is $$K=K_{0}(\mathcal{R}(\mathcal{P}\langle\widetilde{\Delta}(\chi_{0})\rangle),\mathcal{R}(\mathcal{P}\langle\widetilde{\Delta}(\chi_{1})\rangle),\mathcal{R}(\mathcal{P}\langle\widetilde{\Delta}(\chi_{6})\rangle)).$$ By calculation, $K=K_{0}(\omega_{7}+\omega_{7}^{2}+\omega_{7}^{4})$. Since $\forall\sigma\in H$, $\sigma(\omega_{7}+\omega_{7}^{2}+\omega_{7}^{4})=\omega_{7}+\omega_{7}^{2}+\omega_{7}^{4}$, we have $\omega_{7}+\omega_{7}^{2}+\omega_{7}^{4}\in K_{0}$, $K=K_{0}$ and so $\deg(\Gamma)=2$.

\end{example}



\begin{thebibliography}{99}
\small{

\bibitem{dic1}
A. Abdollahi, E. Vatandoost, Which Cayley graphs are integral?, Electron. J. Combin. 16 (2009) \#R122.

\bibitem{int4}
A. Ahmady, J. Bell, B. Mohar, Integral Cayley graphs and groups, SIAM J. Discrete Math. 28 (2014) 685--701.




\bibitem{ncayleylaplacian1}
M. Arezoomand, Normalized Laplacian polynomial of $n$-Cayley graphs, Linear Multilinear Algebra 70 (2022) 2078--2087.

\bibitem{ncayleylaplacian2}
M. Arezoomand, On the Laplacian and signless Laplacian polynomials of graphs with semiregular automorhphisms, J. Algebraic Combin. 52 (2020) 21--32.




\bibitem{tcayleycharacteristic}
M. Arezoomand, B. Taeri, On the characteristic polynomial of $n$-Cayley digraphs, Electron. J. Combin. 20 (2013) \#P57.

\bibitem{semibi0}
M. Arezoomand, B. Taeri, Normality of $2$-Cayley digraphs, Discrete Math. 338 (2015) 41--47.




\bibitem{michaelartin}
M. Artin, Algebra, Second Edition, Pearson, 2010.





\bibitem{bondyandmurty}
J. A. Bondy, U. S. R. Murty, Graph Theory, Springer, New York, 2008.


\bibitem{polycirculant1}
P. J. Cameron, Problems from the fifteenth british combinatorial conference, Discrete Math. 167/168 (1997) 605--615.

\bibitem{polyresult1}
E. Dobson, D. Maru\v{s}i\v{c}, On semiregular elements of solvable groups, Comm. Algebra 39 (2011) 1413--1426.



\bibitem{polyresult2}
E. Dobson, A. Malni\v{c}, D. Maru\v{s}i\v{c}, L. A. Nowitz, Minimal normal subgroups of transitive permutation groups of square-free degree, Discrete Math. 307 (2007) 373--385.


\bibitem{polyresult3}
E. Dobson, A. Malni\v{c}, D. Maru\v{s}i\v{c}, L. A. Nowitz, Semiregular automorphisms of vertex-transitive graphs of certain valencies, J. Combin. Theory Ser. B 97 (2007) 371--380.







\bibitem{semispectrum}
X. Gao, Y. Luo, The spectrum of semi-Cayley graphs over abelian groups, Linear Algebra Appl. 432 (2010) 2974--2983.

\bibitem{semispectrum2}
X. Gao, H. L\"{u}, Y. Hao, The Laplacian and signless Laplacian spectrum of semi-Cayley graphs over abelian groups, J. Appl. Math. Comput. 51 (2016) 383--395.


\bibitem{polyresult4}
M. Giudici, J. Xu, All vertex-transitive locally-quasiprimitive graphs have a semiregular automorphism, J. Algebraic Combin. 25 (2007) 217--232.


\bibitem{normality}
A. Hujdurovi\'{c}, K. Kutnar, D. Maru\v{s}i\v{c}, On normality of $n$-Cayley graphs, Appl. Math. Comput. 332 (2018) 469--476.




\bibitem{dih}
J. Huang, S. Li, Integral and distance integral Cayley graphs over generalized dihedral groups, J. Algebraic Combin. 53 (2021) 921--943.



\bibitem{mixed}
X. Huang, L. Lu, K. M\"{o}nius, Splitting fields of mixed Cayley graphs over abelian groups, J. Algebraic Combin. 58 (2023) 681--693.




\bibitem{fieldandgalois}
N. Jacobson, Basic Algebra I, Second Edition, W.H. Freeman and Company, New York, 1980.


\bibitem{int2}
W. Klotz, T. Sander, Integral Cayley graphs over abelian groups, Electron. J. Combin. 17 (2010) \#R81.

\bibitem{int1}
W. Klotz, T. Sander, Integral Cayley graphs defined by greatest common divisors, Electron. J. Combin. 18 (2011) \#P94.


\bibitem{bisemi1}
H. Koike, I. Kov\'{a}cs, Arc-transitive cubic abelian bi-Cayley graphs and BCI-graphs, Filomat 30 (2016) 321--331.



\bibitem{bisemi2}
I. Kov\'{a}cs, A. Malni\v{c}, D. Maru\v{s}i\v{c}, S. Miklavi\v{c}, One-matching bi-Cayley graphs over Abelian groups, Eur. J. Combin. 30 (2009) 602--616.

\bibitem{polyresult5}
K. Kutnar, D. Maru\v{s}i\v{c}, Recent trends and future directions in vertex-transitive graphs, ARS Math. Contemp. 1 (2008) 112--125.



\bibitem{feili}
F. Li, A method to determine algebraically integral Cayley digraphs on finite abelian group, Contr. Discrete Math. 15 (2020) 148--152.





\bibitem{dicycsemidihe}
W. Liu, J. Tang, J. Wang, J. Yang, Algebraic degree of Cayley graphs over dicyclic and semi-dihedral groups, Appl. Math. Comput. 464 (2024) 128389.





\bibitem{mixed2}
W. Liu, J. Wang, Y. Wu, HS-splitting fields of abelian mixed cayley graphs, Appl. Math. Comput. 456 (2023) 128142.




\bibitem{LiuZ2022}
X. Liu, S. Zhou, Eigenvalues of Cayley graphs, Electron. J. Combin. 29 (2022) \#P2.9, 1--164.








\bibitem{int3}
L. Lu, Q. Huang, X. Huang, Integral Cayley graphs over dihedral groups, J. Algebraic Combin. 47 (2018) 585--601.

\bibitem{lulu}
L. Lu, K. M\"{o}nius, Algebraic degree of Cayley graphs over abelian groups and dihedral groups, J. Algebraic Combin.  57 (2023) 753--761.


\bibitem{bisemi3}
Y. Luo, X. Gao, On the extendability of Bi-Cayley graphs of finite abelian groups, Discrete Math. 309 (2009) 5943--5949.




\bibitem{polycirculant2}
D. Maru\v{s}i\v{c}, On vertex symmetric digraphs, Discrete Math. 36 (1981) 69--81.



\bibitem{fieldandgalois2}
P. Morandi, Field and Galois Theory, Springer-Verlag New York, Inc., 1996.



\bibitem{aldeg2}
K. M\"{o}nius, Splitting fields of spectra of circulant graphs, J. Algebra 594 (2022) 154--169.

\bibitem{aldeg1}
K. M\"{o}nius, The algebraic degree of spectra of circulant graphs, J. Number Theory 208 (2020) 295--304.








\bibitem{aldeg}
K. M\"{o}nius, J. Steuding, P. Stumpf, Which graphs have non-integral spectra?, Graphs Combin. 34 (2018) 1507--1518.





\bibitem{sabidussi}
G. Sabidussi, Vertex-transitive graphs, Monatsh. Math. 68 (1964) 426--438.



\bibitem{so}
W. So, Integral circulant graphs, Discrete Math. 306 (2005) 153--158.



\bibitem{hyper}
N. Sripaisan, Y. Meemark, Algebraic degree of spectra of Cayley hypergraphs, Discrete Appl. Math. 316 (2022) 87--94.


\bibitem{representationtheory}
B. Steinberg, Representation Theory of Finite Groups: An Introduction Approach, Springer Science+Business Media, LLC, 2012.











\bibitem{polyresult6}
G. Verret, Arc-transitive graphs of valency 8 have a semiregular automorphism, ARS Math. Contemp. 8 (2015) 29--34.


\bibitem{quasiabelian}
S. Wang, M. Arezoomand, T. Feng, Algebraic degrees of quasi-abelian semi-Cayley digraphs, \url{arXiv:2308.03066}, 2023.




\bibitem{2cayley}
Y. Wu, Y. Jing, L. Feng, Algebraic degrees of $2$-Cayley digraphs over abelian groups, Ars Math. Contemp. 24 (2024) \#P2.02.



\bibitem{semibi1}
J. X. Zhou, Every finite group has a normal bi-Cayley graph, ARS Math. Contemp. 14 (2018) 177--186.


\bibitem{bisemi4}
J. X. Zhou, Y. Q. Feng, Cubic non-Cayley vertex-transitive bi-Cayley graphs over a regular $p$-group, Electron. J. Combin. 23 (2016) \#P3.34. 


\bibitem{bisemi5}
J. X. Zhou, Y. Q. Feng, The automorphisms of bi-Cayley graphs, J. Combin. Theory Ser. B 116 (2016) 504--532.


\bibitem{bisemi6}
X. Zou, J. X. Ment, Some algebraic properties of Bi-Cayley graphs, Acta Math. Sin. (Chin. Ser.) 50 (2007) 1075--1079.












































%























%

































}

\end{thebibliography}
\end{document}